\documentclass[11pt]{amsart}
\usepackage{amssymb}
\usepackage{amscd}
\usepackage{mathtools}

\issueinfo{00}{0}{Month}{Year}
\copyrightinfo{2008}{University of Calgary}
\pagespan{1}{2}
\PII{ISSN 1715-0868}

\usepackage{graphicx}
\usepackage{epstopdf}
\usepackage{amsthm,amsmath,amssymb,amsxtra,amscd}
\usepackage{verbatim}
\usepackage{indent}
\usepackage{rotating}
\usepackage{multirow}
\usepackage{multicol}
\usepackage{url}
\usepackage{algorithmic}
\usepackage{xypic}

\newtheorem{theorem}{Theorem}[section]

\newtheorem{lemma}[theorem]{Lemma}
\newtheorem{corollary}[theorem]{Corollary}

\newtheorem{proposition}[theorem]{Proposition}

\newtheorem*{example*}{Example}

\newtheoremstyle{myexample}{3pt}{3pt}{\rmfamily}{}{\itshape}{:}{ }{\thmname{#1}\thmnumber{ #2}\thmnote{ (#3)}}
\theoremstyle{myexample}

\newtheoremstyle{myremark}{3pt}{3pt}{\rmfamily}{}{\itshape}{:}{ }{\thmname{#1}}
\theoremstyle{myremark}
\newtheorem{remark}[theorem]{Remark}
\newtheorem*{observation*}{Observation}

\newtheoremstyle{conjecture}{3pt}{3pt}{\itshape}{}{\bfseries}{.}{ }{\thmname{#1}\thmnote{ (#3)}}
\theoremstyle{conjecture}

\newtheorem*{question*}{Question}
\newtheorem{conjecture}{Conjecture}
\newtheorem{theorem*}{Theorem}

\numberwithin{equation}{section}

\newcounter{algorithm}
\renewcommand{\thealgorithm}{\thesection.\arabic{algorithm}}

\newcommand\R{{\mathbf{R}}}
\newcommand\C{{\mathbf{C}}}
\newcommand\Z{{\mathbf{Z}}}

\newcommand\eps{{\varepsilon}}

\begin{document}

\title[Lonely runner conjecture]{Some remarks on the lonely runner conjecture}

\author{Terence Tao}
\address{UCLA Department of Mathematics, Los Angeles, CA 90095-1555.}
\email{tao@math.ucla.edu}

%

\date{}

\thanks{The author is supported by NSF grant DMS-1266164, the James and Carol Collins Chair, and by a Simons Investigator Award.  We thank the anonymous referee for many suggestions and corrections, and Anthony Quas, Georges Grekos, and Oriol Serra for further corrections.}

\begin{abstract}
The \emph{lonely runner conjecture} of Wills and Cusick, in its most popular formulation, asserts that if $n$ runners with distinct constant speeds run around a unit circle $\R/\Z$ starting at a common time and place, then each runner will at some time be separated by a distance of at least $\frac{1}{n+1}$ from the others.  In this paper we make some remarks on this conjecture.  Firstly, we can improve the trivial lower bound of $\frac{1}{2n}$ slightly for large $n$, to $\frac{1}{2n} + \frac{c \log n}{n^2 (\log\log n)^2}$ for some absolute constant $c>0$; previous improvements were roughly of the form $\frac{1}{2n} + \frac{c}{n^2}$.  Secondly, we show that to verify the conjecture, it suffices to do so under the assumption that the speeds are integers of size $n^{O(n^2)}$.  We also obtain some results in the case when all the velocities are integers of size $O(n)$.
\end{abstract}

\maketitle

\section{Introduction}

The \emph{lonely runner conjecture} of Wills \cite{wills} and Cusick \cite{cusick} (as formulated in \cite{bggst}) asserts that if $n \geq 2$ is an integer and $n$ runners run around the unit circle $\R/\Z$ with constant distinct speeds starting from a common time and place, then each runner is ``lonely'' in the sense that there exists a time in which the runner is separated by a distance at least $\frac{1}{n}$ from the others.  The conjecture originated from questions in view obstruction \cite{cusick} and diophantine approximation \cite{wills}, but also has connections to chromatic numbers of distance graphs \cite{zhu} and to flows in regular matroids \cite{bggst}.  The conjecture is known for $n \leq 7$ (see \cite{bs} and the references therein), under various ``lacunarity'' hypotheses on the velocities (see \cite{pandey}, \cite{rtv}, \cite{bs2}, \cite{dub}), or if one is allowed to ignore one runner of one's choosing at any given time \cite{cg}.  We refer the reader to the recent paper \cite{ps} for further discussion of the literature on this conjecture and additional references.

It is known (see e.g. \cite[\S 4]{bhk}) that one can assume without loss of generality that the speeds of the runners are integers, which allows one to place the time variable $t$ in the unit circle $\R/\Z$ rather than on the real line; one can also normalise the speed of the ``lonely'' runner to be zero.  This allows us to reformulate the conjecture (after decrementing $n$ by one to account for the normalised speed of the lonely runner) as follows.  Given an element $t$ of the unit circle $\R/\Z$, let $\|t\|_{\R/\Z}$ denote the distance of (any representative of) $t$ to the nearest integer.  Given an $n$-tuple of non-zero integers $v_1,\dots,v_n$, let $\delta(v_1,\dots,v_n)$ denote the maximal value of $\min(\|tv_1\|_{\R/\Z}, \dots, \|tv_n\|_{\R/\Z})$ as $t$ ranges in $\R/\Z$; note that this minimum is attained because $\R/\Z$ is compact.  We then let $\delta_n$ denote the infimal value of $\delta(v_1,\dots,v_n)$ as $(v_1,\dots,v_n)$ ranges over $n$-tuples of distinct non-zero integers; this quantity was termed the \emph{gap of loneliness} in \cite{ps}.  The Dirichlet approximation theorem implies that $\delta(1,\dots,n) \leq \frac{1}{n+1}$, and hence 
\begin{equation}\label{deltan}
\delta_n \leq \frac{1}{n+1}.
\end{equation}
See also \cite{gw} for further sets of $n$-tuples $(v_1,\dots,v_n)$ that witness this bound.
The lonely runner conjecture is then equivalent to the assertion that this bound is sharp:

\begin{conjecture}[Lonely runner conjecture]\label{lrc}  For every $n \geq 1$, one has $\delta_n = \frac{1}{n+1}$.
\end{conjecture}

Because we have decremented $n$ by one, Conjecture \ref{lrc} is currently only known for $n \leq 6$ \cite{bs}.

For any \emph{frequency} $v \in \Z$ and radius $\delta>0$, we define the \emph{rank one Bohr set}
\begin{equation}\label{bohr-def}
 B(v; \delta) \coloneqq \{ t \in \R/\Z: \|tv\|_{\R/\Z} \leq \delta \};
\end{equation}
more generally, we define the higher rank Bohr sets $B(v_1,\dots,v_r; \delta_1,\dots,\delta_r)$ for $\delta_1,\dots,\delta_r > 0$ and $v_1,\dots,v_r \in \Z$ and some \emph{rank} $r \geq 1$ by the formula
\begin{equation}\label{bohr2-def}
B(v_1,\dots,v_r; \delta_1,\dots,\delta_r) \coloneqq B(v_1;\delta_1) \cap \dots \cap B(v_r;\delta_r).
\end{equation}
We can then interpret $\delta_n$ in terms of Bohr sets in a number of equivalent ways:
\begin{itemize}
\item[(i)]  $\delta_n$ is the largest number for which one has the strict inclusion
$$ \bigcup_{i=1}^n B(v_i; \delta) \subsetneq \R/\Z$$
(or equivalently, $\min(\|tv_1\|_{\R/\Z}, \dots, \|tv_n\|_{\R/\Z}) > \delta$ for some time $t$) for every non-zero integers $v_1,\dots,v_n$ and $0 < \delta < \delta_n$.
\item[(ii)] Taking contrapositives, $\delta_n$ is the least number for which one there exists a covering of the form
\begin{equation}\label{rcover}
 \R/\Z = \bigcup_{i=1}^n B(v_i; \delta_n)
\end{equation}
of the unit circle by $n$ \emph{rank one Bohr sets} $B(v_i; \delta_n)$, $i=1,\dots,n$, for some non-zero integers $v_1,\dots,v_n$.
\end{itemize}

We have a simple and well known lower bound on $\delta_n$ that gets within a factor of two of the lonely runner conjecture:

\begin{proposition}\label{pdn}  For every $n \geq 1$, one has $\delta_n \geq \frac{1}{2n}$.
\end{proposition}

\begin{proof}  If $t$ is drawn uniformly at random from $\R/\Z$, then for any nonzero integer $v_i$, $tv_i$ is also distributed uniformly at random on $\R/\Z$.  Letting $m$ denote Lebesgue measure on $\R/\Z$, we thus have
\begin{equation}\label{mvd}
 m( B(v_i;\delta_n) ) = 2 \delta_n 
\end{equation}
for every $i=1,\dots,n$ (noting from \eqref{deltan} that $\delta_n \leq 1/2$).  Using the union bound
\begin{equation}\label{union}
m\left( \bigcup_{i=1}^n B(v_i;\delta_n) \right) \leq \sum_{i=1}^n m( B(v_i;\delta_n) )
\end{equation}
and \eqref{rcover}, we conclude that
$$ 1 \leq \sum_{i=1}^n 2 \delta_n$$
and the claim follows.
\end{proof}

The union bound \eqref{union} is very crude, and one would naively expect to be able to improve significantly upon Proposition \ref{pdn} by using more sophisticated bounds, for instance by using some variant of the inclusion-exclusion formula combined with bounds on the size of higher-rank Bohr sets $B(v_{i_1},\dots,v_{i_r};\delta_n,\dots,\delta_n)$.  However, only slight improvements to this bound are known.  Chen \cite{chen} obtained the bound
\begin{equation}\label{dd}
 \delta_n \geq \frac{1}{2n - 1 + \frac{1}{2n-3}}
\end{equation}
and Chen and Cusick \cite{cc} obtained the improvement
$$ \delta_n \geq \frac{1}{2n-3}$$
assuming that $2n-3$ was prime.  In the recent paper \cite{ps}, a bound of the form
\begin{equation}\label{ee}
 \delta_n \geq \frac{1}{2n - 2 + o(1)}
\end{equation}
as well as the variant bound
\begin{equation}\label{variant}
\delta(v_1,\dots,v_n) \geq \frac{1}{2(n-\sum_{i=2}^n \frac{1}{v_i})}
\end{equation}
was obtained as $n \to \infty$, without any primality restrictions.  These improvements relied primarily on estimates on rank two Bohr sets $B(v_i,v_j;\delta_n,\delta_n)$.

These bounds only improve on the bound in Proposition \ref{pdn} by a multiplicative factor of $1 + O(\frac{1}{n})$.  The following example can help explain why this factor is so close to $1$.  Let $n$ be a large integer, and let $p_1,\dots,p_s$ denote the primes between $n/4$ and $n/2$, thus by the prime number theorem $s = (1+o(1)) \frac{n}{4 \log n}$ as $n \to \infty$.  For each $i=1,\dots,s$, the rank one Bohr set $B(p_i; \delta_n)$ consists of $p_i$ intervals of the form $[\frac{a}{p_i}-\frac{\delta_n}{p_i}, \frac{a}{p_i} + \frac{\delta_n}{p_i}]$ for $a=0,\dots,p-1$ (where we identify these intervals with subsets of $\R/\Z$ in the usual fashion).  This makes this collection of Bohr sets behave like a ``sunflower'' (in the sense of \cite{er}) with a very small ``kernel''.  To see this, we separate the $a=0$ interval of $B(p_i; \delta_n)$ from the others, writing
$$ B(p_i;\delta_n) = \left[-\frac{\delta_n}{p_i}, \frac{\delta_n}{p_i}\right] \cup B'(p_i;\delta_n)$$
where we think of the interval $\left[-\frac{\delta_n}{p_i}, \frac{\delta_n}{p_i}\right]$ as the ``kernel'' of $B(p_i, \delta_n)$, and where $B'(p_i,\delta_n)$ is the ``petal'' set
$$ B'(p_i;\delta_n) \coloneqq \bigcup_{a=1}^{p_i-1} \left[\frac{a}{p_i}-\frac{\delta_n}{p_i}, \frac{a}{p_i} + \frac{\delta_n}{p_i}\right].$$
Clearly, the interval $[-\frac{\delta_n}{p_i}, \frac{\delta_n}{p_i}]$ has measure $\frac{2\delta_n}{p_i}$, and so by \eqref{mvd} the remaining portion $B'(p_i;\delta_n)$ of the rank one Bohr set has measure $(1 - \frac{1}{p_i}) 2\delta_n$.  

Now we claim that the ``petal'' sets $B'(p_i;\delta_n)$ for $i=1,\dots,r$ are disjoint, for reasons relating to the spacing properties of the Farey sequence.  Indeed, suppose for contradiction that there was a point $t \in \R/\Z$ that was in both $B'(p_i;\delta_n)$ and $B'(p_j;\delta_n)$ for some $1 \leq i < j \leq n$.  Then we have
$$ \left\| t - \frac{a}{p_i}\right \|_{\R/\Z} \leq \frac{\delta_n}{p_i}, \quad \left\| t - \frac{b}{p_j}\right \|_{\R/\Z} \leq \frac{\delta_n}{p_j} $$
for some $1 \leq a \leq p_i$ and $1 \leq b \leq p_j$.  In particular by the triangle inequality we have
$$ \left\| \frac{a}{p_i} - \frac{b}{p_j} \right\|_{\R/\Z} \leq \frac{\delta_n}{p_i} + \frac{\delta_n}{p_j}.$$
On the other hand, as $p_i,p_j$ are distinct primes, and $a,b$ are not divisible by $p_i,p_j$ respectively, the fraction $\frac{a}{p_i} - \frac{b}{p_j}$ is not an integer, and hence
$$ \left\| \frac{a}{p_i} - \frac{b}{p_j} \right\|_{\R/\Z} \geq \frac{1}{p_i p_j}.$$
Comparing the two inequalities and multiplying by $p_i p_j$, we obtain
$$ \delta_n (p_i + p_j) \geq 1,$$
but this contradicts \eqref{deltan} and the hypothesis $p_i, p_j \leq n/2$.

From this disjointness, we see that the union bound is obeyed with equality for the $B'(p_i,\delta_n)$, and hence
\begin{align*}
m\left( \bigcup_{i=1}^s B(p_i;\delta_n) \right) &\geq m\left( \bigcup_{i=1}^s B'(p_i;\delta_n)\right) \\
&= \sum_{i=1}^s m(B'(p_i;\delta_n)) \\
&= \sum_{i=1}^s \left(1-\frac{1}{p_i}\right) 2 \delta_n\\
&\geq \left(1-\frac{4}{n}\right) \sum_{i=1}^s m( B(p_i;\delta_n) ).
\end{align*}
In particular, we see that the union bound 
$$ m\left( \bigcup_{i=1}^s B(p_i;\delta_n) \right) \leq \sum_{i=1}^s m( B(p_i;\delta_n) )$$
is only off from the truth by a multiplicative factor of $1 + O(\frac{1}{n})$, which is consistent with the improvements to Proposition \ref{pdn} in the known literature.

On the other hand, the above example only involves $s$ rank one Bohr sets rather than $n$ rank one Bohr sets.  As $s$ is comparable to $n/\log n$ rather than $n$, this suggests that perhaps some ``logarithmic'' improvement to the known lower bounds on $\delta_n$ is still possible via some refinement of the union bound.  The first main result of this paper shows that this is (almost) indeed the case:

\begin{theorem}\label{main1}  There exists an absolute constant $c>0$ such that
$$ \delta_n \geq \frac{1}{2n} + \frac{c \log n}{n^2 (\log\log n)^2}$$
for all sufficiently large $n$.
\end{theorem}

It is likely that with a refinement of the arguments below, one could eliminate at least one of the $\log\log n$ factors in the denominator; however the example discussed above suggests to the author that significantly more effort would be needed in order to improve the $\log n$ factor in the numerator by these methods.  The constant $c$ is in principle computable explicitly, but we have not attempted to arrange the arguments to optimise this constant.

We prove Theorem \ref{main1} in Section \ref{first-sec}.  In addition to the control on rank two Bohr sets $B(v_i,v_j; \delta_n,\delta_n)$ that was exploited in previous literature, we also now use estimates on the size of rank three Bohr sets $B(v_i,v_j,v_k;\delta_n,\delta_n,\delta_n)$.  The key point is that if the union bound \eqref{union} were to be close to sharp with $\delta_n$ very close to $1/2n$, then one can use H\"older's inequality (or the Cauchy-Schwarz inequality), together with lower bounds on the size of rank two Bohr sets to show that many rank three Bohr sets must be extremely large.  After using some Fourier analysis to compute the size of these rank three Bohr sets, together with some elementary additive combinatorics involving generalised arithmetic progressions, one eventually concludes that a large fraction of the velocities $v_i$ must be\footnote{This can be compared with the results in \cite{czer}, \cite{alon}, which study the opposite case where the velocities are assumed to be random rather than highly structured, in which case the gap $\delta(v_1,\dots,v_n)$ is in fact very close to $1/2$.} essentially contained (ignoring some ``small denominators'') in an arithmetic progression of length comparable to $n$ and symmetric around the origin.  As the preceding example indicates, this by itself is not inconsistent with the union bound being close to tight, if the velocities $v_i$ behave like (rescaled versions of) prime numbers $p_i$.  But the primes are a logarithmically sparse set, and standard sieve theory bounds tell us that most numbers of size comparable to $n$ will not only be composite, but in fact contain a medium-sized prime factor (e.g. a factor between $\log^{10} n$ and $n^{1/10}$).  One can use these medium-sized prime factors to show that many of the rank one Bohr sets will intersect other rank one Bohr sets in various disjoint (and reasonably large) ``major arcs'', which can then be used to improve upon the union bound.  See also \eqref{variant} for some comparable improvements on the union bound in the case when the velocities are contained in a progression of length comparable to $n$.

Our second result is of a different nature, and is concerned with the decidability of the lonely runner conjecture for bounded values of $n$.  In its current formulation, it is not obvious that one can decide Conjecture \ref{lrc} in finite time for any fixed $n$, since one potentially has to compute $\delta(v_1,\dots,v_n)$ for an infinite number of tuples $(v_1,\dots,v_n)$.  However, the following result\footnote{This result first appeared on the author's blog at {\tt terrytao.wordpress.com/2015/05/13}.} shows that one only needs to verify the conjecture for a finite (albeit large) number of tuples for each $n$:

\begin{theorem}\label{main2}  There exists an absolute (and explicitly\footnote{This theorem is trivially true if one allows the constant $C_0$ to be ineffective.  Indeed, one could set $C_0$ to be arbitrary if the lonely runner conjecture was true up to $n_0$, and to be sufficiently large (depending on the first counterexample to this conjecture) otherwise.  We thank Kevin O'Bryant for this remark.} computable) constant $C_0>0$, such that the following assertions are logically equivalent for every natural number $n_0 \geq 1$:
\begin{itemize}
\item[(i)]  One has $\delta_n = \frac{1}{n+1}$ for all $n \leq n_0$ (that is, Conjecture \ref{lrc} holds for $n$ up to $n_0$).
\item[(ii)]  One has $\delta(v_1,\dots,v_n) \geq \frac{1}{n+1}$ for all $n \leq n_0$ and every tuple $(v_1,\dots,v_n)$ of non-zero distinct integers with $|v_i| \leq n^{C_0 n^2}$ for all $i=1,\dots,n$.
\end{itemize}
\end{theorem}

Since $\delta(v_1,\dots,v_n)$ is clearly computable for any fixed choice of $v_1,\dots,v_n$ (note that the function $\min(\|t v_1\|_{\R/\Z},\dots,\|t v_n\|_{\R/\Z})$ is piecewise linear), we conclude

\begin{corollary}  For any natural number $n_0 \geq 1$, the assertion that Conjecture \ref{lrc} holds for $n$ up to $n_0$ is decidable (and the truth value may be computed in time $O(n_0^{O(n_0^2)})$).
\end{corollary}

Thus, for instance, one can decide in finite time whether Conjecture \ref{lrc} holds for $n=7$.  Unfortunately, the bounds on $v_1,\dots,v_n$ given by the above theorem are far too large to suggest a practical algorithm for doing so.  On the other hand, it is not clear that the previous work on the lonely runner conjecture for small values of $n$ could extend, even in principle, to all larger values of $n$; for instance, the arguments in \cite{bs} that treated the $n=6$ case relied crucially on the fact that $n+1$ was prime.

We prove Theorem \ref{main2} in Section \ref{second-sec}.  Roughly speaking, the argument proceeds as follows.  The implication of (ii) from (i) is trivial; the main task is to show that (ii) implies (i).  That is, we have to use (ii) to obtain the bound
\begin{equation}\label{dvn}
\delta(v_1,\dots,v_n) \geq \frac{1}{n+1} 
\end{equation}
for \emph{every} choice of non-zero distinct integers $(v_1,\dots,v_n)$, and all $n \leq n_0$.  Using standard additive combinatorics, one can place the velocities $v_1,\dots,v_n$ somewhat efficiently in a ``sufficiently proper'' generalised arithmetic progression $P = P( w_1,\dots,w_r; N_1,\dots,N_r )$ of some rank $r \geq 1$.  If this rank is equal to one, then we can easily derive \eqref{dvn} from (ii) by a rescaling argument.  If the rank exceeds one, then it is possible to map (via a \emph{Freiman homomorphism}, see e.g. \cite{tv}) the velocities $v_1,\dots,v_n$ to a transformed set of non-zero velocities $v'_1,\dots,v'_n$ for which there is at least one \emph{collision} $v'_i = v'_j$ for some distinct $i,j$.  Assuming inductively that \eqref{dvn} has already been established for smaller values of $n$, we can show that
\begin{equation}\label{dvn2}
\delta(v'_1,\dots,v'_n) \geq \frac{1}{n} 
\end{equation}
One then uses Fourier analysis to show (if $P$ is sufficiently proper, and the Freiman homomorphism is of sufficiently high quality) that $\delta(v_1,\dots,v_n)$ is close to (or larger than) $\delta(v'_1,\dots,v'_n)$; because $\frac{1}{n}$ is slightly larger than $\frac{1}{n+1}$, this will let us obtain \eqref{dvn} after selecting all the quantitative parameters suitably.

Theorem \ref{main2} suggests that one possible route to solving Conjecture \ref{lrc} is to first reduce to the case when $v_1,\dots,v_n$ are contained in a fairly short progression, and then treat that case by a separate argument.  As a simple example of such an argument, we prove the following elementary (but rather weak) result in Section \ref{short-sec}:

\begin{proposition}\label{short}  Let $n \geq 1$, and let $v_1,\dots,v_n$ be positive integers such that $v_i \leq 1.2 n$ for all $i=1,\dots,n$.  Then
$\delta(v_1,\dots,v_n) \geq \frac{1}{n+1}$.
\end{proposition}

There is of course a huge gap between $1.2n$ and $n^{C_0 n^2}$, and so Theorem \ref{main2} and Proposition \ref{short} fall well short of a full proof of the lonely runner conjecture.  Nevertheless it seems of interest to increase the quantity $1.2n$ appearing in Proposition \ref{short}.  A natural target would be $2n$, as one then has multiple (presumed) extremisers $\delta(v_1,\dots,v_n) = \frac{1}{n+1}$ even after accounting for the freedom to permute the $v_1,\dots,v_n$.  Indeed, in addition to the standard extremiser $(1,2,\dots,n)$, one now also has the dilate $(2,4,\dots,2n)$, and also one has a number of additional examples coming from the construction in \cite{gw}, namely those tuples formed from $(1,\dots,n)$ by replacing one element $r \in \{2,\dots,n-1\}$ with $2r$, provided that $r$ shares a common factor with each integer $b$ in the range $n-r+1 \leq b \leq 2n-2r+1$; for instance one can take $r=6, n=7$ and consider\footnote{This particular tuple was also discovered previously in an unpublished work of Peter Flor.  Thanks to J\"org Wills for this reference.  See also \cite{gw} for some other variants of this construction, for instance one can take the tuple $(1,2,\dots,73)$ and replace the two velocities $70,72$ by their doubles $140, 144$ respectively and still obtain an extremiser.} the tuple $(1,2,3,4,5,7,12)$.  

In a similar spirit to Proposition \ref{short}, we have the following improvement of Theorem \ref{main1} when the $v_i$ are constrained to be small:

\begin{proposition}\label{short2}  Let $C > 0$, and suppose that $n$ is sufficiently large depending on $C$.  Let $v_1,\dots,v_n$ be positive integers such that $v_i \leq Cn$ for all $i=1,\dots,n$.  Then one has $\delta(v_1,\dots,v_n) \geq \frac{1+c}{2n}$, where $c>0$ depends only on $C$.
\end{proposition}

We prove this proposition in Section \ref{short-sec}, using some of the same arguments used to prove Theorem \ref{main1}.  This bound may be compared with the bound \eqref{variant}, which under the hypotheses of Proposition \ref{short} can give a bound slightly better than that in Theorem \ref{main1} if the $v_i$ contain some relatively small integers.


\section{Notation and preliminaries}

In this paper, $n$ will be an asymptotic integer parameter going to infinity.  We use the notation $X = O(Y)$, $X \ll Y$, or $Y \gg X$ to denote an estimate of the form $|X| \leq CY$ where $C$ is independent of $n$.  In some cases, $C$ will be allowed to depend on other parameters, and we will denote this by subscripts unless otherwise specified, for instance $X \ll_r Y$ means that $|X| \leq C_r Y$ for some $C_r$ depending on $r$.
We use $X \asymp Y$ to denote the estimates $X \ll Y \ll X$.  Thus for instance from \eqref{deltan} and Proposition \ref{pdn} we have
\begin{equation}\label{delta-comp}
\delta_n \asymp \frac{1}{n}.
\end{equation}
We also use $X = o(Y)$ to denote the estimate $|X| \leq c(n) Y$ where $c(n)$ is a quantity that goes to zero as $n \to \infty$, keeping all other parameters independent of $n$ fixed.

Given a finite set $S$, we use $\# S$ to denote its cardinality. Given a statement $E$, we define the indicator $1_E$ to be $1$ if $E$ is true and $0$ if $E$ is false.  If $A$ is a set, we write $1_A$ for the indicator function $1_A(x) \coloneqq 1_{x \in A}$.

Given a function $\phi \colon \R \to \R$ which is in the Schwartz class (that is, smooth and all derivatives rapidly decreasing), its \emph{Fourier transform} $\hat \phi \colon \R \to \R$ is defined by the formula
$$ \hat \phi(t) \coloneqq \int_\R \phi(s) e^{-2\pi i ts}\ ds;$$
as is well known, this is also in the Schwartz class, and we have the Fourier inversion formula
$$ \phi(s) \coloneqq \int_\R \hat \phi(t) e^{2\pi i ts}\ dt.$$
One can construct $\phi$ which are non-negative and compactly supported, and whose Fourier transform $\hat \phi$ is strictly positive; indeed, starting from any non-negative compactly supported and even $\phi$, one can convolve $\phi$ with itself to make $\hat \phi$ non-negative, and then square the resulting convolution to make $\hat \phi$ strictly positive everywhere.  Inverting the Fourier transform, one can also find a Schwartz class $\phi$ that is strictly positive, and whose Fourier transform is non-negative and compactly supported; by rescaling, one can make this compact support as small as desired (e.g. contained in $[-1,1]$).

In addition to the Bohr sets $B(v_1,\dots,v_r;\delta_1,\dots,\delta_r)$ defined in \eqref{bohr-def}, we will also need the dual notion of a \emph{generalised arithmetic progression} $P( w_1,\dots,w_r; N_1,\dots,N_r)$ of some rank $r \geq 1$, defined for \emph{generators} $w_1,\dots,w_r \in \Z$ and \emph{dimensions} $N_1,\dots,N_r > 0$ (which may be real numbers instead of integer) as
$$ P(w_1,\dots,w_r; N_1,\dots,N_r) \coloneqq \{ n_1 w_1 + \dots + n_r w_r: n_1,\dots,n_r \in \Z; |n_i| \leq N_i\ \forall i=1,\dots, n \}.$$
Given such a generalised arithmetic progression $P = P(w_1,\dots,w_r; N_1,\dots,N_r)$ and a scaling factor $t>0$, we define the dilation\footnote{Strictly speaking, for this notation to be well-defined, one should view $P$ not just as an unstructured set of integers, but as a tuple $(r, (w_1,\dots,w_r), (N_1,\dots,N_r), P(w_1,\dots,w_r;N_1,\dots,N_r))$, because it is possible for a single set of integers to arise from progressions of different ranks, generators, and dimensions, and the dilations $tP$ may depend on this data.  Similarly for the notion of $t$-properness.  However, we shall abuse notation and identify a generalised arithmetic progression with the set of its elements.} $tP$ to be the generalised arithmetic progression
$$ tP \coloneqq P(w_1,\dots,w_r; tN_1,\dots,tN_r),$$
in particular
$$ 2P \coloneqq P(w_1,\dots,w_r; 2N_1,\dots,2N_r).$$
A generalised arithmetic progression $P = P(w_1,\dots,w_r; N_1,\dots,N_r)$ is said to be \emph{$t$-proper} if the sums $n_1 w_1 + \dots + n_r w_r$ for $n_1,\dots,n_r \in \Z$ and $|n_i| \leq t N_i$ for $i=1,\dots,n$ are all distinct.  Thus for instance any rank one progression $P(w,N)$ will be $t$-proper for any $t>0$ if the generator $w$ is non-zero.  For ranks greater than one, it is possible for generalised arithmetic progressions to fail to be $t$-proper even when the generators are non-zero; indeed, this is inevitable for $t$ large enough.  However, we do have the following inclusion:

\begin{proposition}[Progressions lie in proper progressions]\label{ppp}  Let $P = P(w_1,\dots,w_r; N_1,\dots,N_r)$ be a generalised arithmetic progression, and let $t \geq 1$.  Then there exists a generalised $t$-proper arithmetic progression $Q = Q(w'_1,\dots,w'_{r'}; N'_1,\dots,N'_{r'})$ with $r' \leq r$ such that $P \subset Q$ and
$$\# Q \leq (2t)^r r^{6r^2} \prod_{i=1}^r (2N_i+1).$$
\end{proposition}

\begin{proof}  See \cite[Theorem 2.1]{green} (see also \cite{bilu}).
\end{proof}

Given a generalised arithmetic progression $P = P(v_1,\dots,v_r; N_1,\dots,N_r)$, define its \emph{multiplicity} $\mu(P)$ to be the number of tuples $(n_1,\dots,n_r) \in \Z^r$ with $|n_i| \leq N_i$ for $i=1,\dots,r$ such that $n_1 v_1 + \dots n_r v_r = 0$.  Thus $\mu(P)$ is a positive integer that equals $1$ when $P$ is $1$-proper; conversely, for any $t > 0$, $P$ will be $t$-proper whenever $\mu(2tP) = 1$.  We have the following basic connection between the size of a Bohr set $B(v_1,\dots,v_r; \delta_1,\dots,\delta_r)$ and the multiplicity of its dual progression $P(v_1,\dots,v_r; \frac{1}{\delta_1},\dots,\frac{1}{\delta_r})$:

\begin{lemma}[Size of Bohr sets]\label{lod}  For any Bohr set $B(v_1,\dots,v_r; \delta_1,\dots,\delta_r)$ with $\delta_1,\dots,\delta_r < 1/2$, one has
$$ 
m( B( v_1,\dots,v_r; \delta_1,\dots,\delta_r) = \exp(O(r)) \mu\left(P\left(v_1,\dots,v_r; \frac{1}{\delta_1},\dots,\frac{1}{\delta_r}\right)\right) \prod_{j=1}^r \delta_r.$$
\end{lemma}

\begin{proof} 
Let $\Omega$ denote the set of tuples $(n_1,\dots,n_r) \in \Z^r$ with $|n_j| \leq 1/\delta_j$ for $j=1,\dots,r$, such that
$$ n_1 v_1 + \dots + n_r v_r = 0,$$
then our task is to show that
$$ 
m( B( v_1,\dots,v_r; \delta_1,\dots,\delta_r) =  \exp(O(r)) (\# \Omega) \prod_{j=1}^r \delta_r.$$

We first prove the upper bound.  As discussed previously, we can locate a Schwartz class function $\phi \colon \R \to \R$ supported on $[-1,1]$ whose Fourier transform $\hat \phi$ is positive everywhere; we allow implied constants to depend on $\phi$.  Then by the Poisson summation formula we have
\begin{align*}
m\left( \bigcap_{j=1}^r B(v_j,\delta_j) \right) &= \int_0^1 \prod_{j=1}^r \sum_{m_j \in \Z} 1_{[-\delta_j,\delta_j]}( t v_j + m_j )\ dt \\
&\leq  \exp(O(r)) \int_0^1 \prod_{j=1}^r \sum_{m_j \in \Z} \hat \phi\left( \frac{t v_j + m_j}{\delta_j} \right)\ dt \\
&= \exp(O(r))\int_0^1 \prod_{j=1}^r \sum_{n_j \in \Z} \delta_j \phi( \delta_j n_j ) e^{-2\pi i t v_j n_j}\ dt \\
&=  \exp(O(r))\sum_{(n_1,\dots,n_r) \in \Omega} \prod_{j=1}^r \delta_j \phi( \delta_j n_j )  \\
&\leq  \exp(O(r)) (\# \Omega)\prod_{j=1}^r \delta_j
\end{align*}
as required.  For the lower bound, we swap the roles of $\phi$ and $\hat \phi$:
\begin{align*}
m( \bigcap_{j=1}^r B(v_j,\delta_j) ) &= \int_0^1 \prod_{j=1}^r \sum_{m_j \in \Z} 1_{[-\delta_j,\delta_j]}( t v_j + m_j )\ dt \\
&\geq  \exp(O(r)) \int_0^1 \prod_{j=1}^r \sum_{m_j \in \Z} \phi\left( \frac{t v_j + m_j}{\delta_j} \right)\ dt \\
&=  \exp(O(r))\int_0^1 \prod_{j=1}^r \sum_{n_j \in \Z} \delta_j \hat \phi( \delta_j n_j ) e^{2\pi i t v_j n_j}\ dt \\
&=  \exp(O(r))\sum_{(n_1,\dots,n_r) \in \Z^r: n_1 v_1 + \dots + n_r v_r = 0} \prod_{j=1}^r \delta_j \hat \phi( \delta_j n_j ) \\
&\geq  \exp(O(r)) (\# \Omega)\prod_{j=1}^r \delta_j.
\end{align*}
\end{proof}

Combining the above lemma with the crude lower bound
$$ \mu\left(P\left(v_1,\dots,v_r; \frac{1}{\delta_1},\dots,\frac{1}{\delta_r}\right)\right) \geq 1$$
we obtain (cf. \cite[Lemma 4.20]{tv})

\begin{corollary}[Crude lower bound on Bohr set size]\label{clb}
For any Bohr set $B(v_1,\dots,v_r; \delta_1,\dots,\delta_r)$, one has
$$ 
m( B( v_1,\dots,v_r; \delta_1,\dots,\delta_r) \geq \exp(O(r)) \prod_{j=1}^r \delta_r.
$$
\end{corollary}

For further discussion of generalised arithmetic progressions and Bohr sets, see \cite{tv}.

\section{Proof of first theorem}\label{first-sec}

We now prove Theorem \ref{main1}.  Suppose the claim failed, then there exist arbitrarily large $n$ for which one has
$$ \delta_n \leq \frac{1}{2n} + o\left( \frac{\log n}{n^2 (\log\log n)^2} \right),$$
thus (by the existing bounds on $\delta_n$) one has
\begin{equation}\label{dnd}
 \delta_n = \frac{1}{2n} + \frac{A}{n^2}
\end{equation}
for some $A$ with
\begin{equation}\label{asm}
A = o\left( \frac{\log n}{(\log\log n)^2} \right).
\end{equation}
Using one of the existing bounds \eqref{dd}, \eqref{ee} we also see that 
\begin{equation}\label{a-large}
A \gg 1
\end{equation}
for $n$ large enough.

Henceforth we assume that $n$ is sufficiently large and that \eqref{asm}, \eqref{a-large} holds. In particular
\begin{equation}\label{dnd-2}
\delta_n = \frac{1+o(1)}{2n}.
\end{equation}
By definition, we can find $v_1,\dots,v_n$ such that $\delta(v_1,\dots,v_n)$ is arbitrarily close to $\delta_n$.  To simplify the notation, we shall assume we can find $v_1,\dots,v_n$ such that $\delta(v_1,\dots,v_n) = \delta_n$; the general case follows by adding an arbitrarily small error to the arguments below.

If we let $F \colon \R/\Z \to \R$ denote the multiplicity function
$$ F(t) \coloneqq \sum_{i=1}^n 1_{B(v_i,\delta_n)}(t)$$
then from \eqref{rcover} we have 
\begin{equation}\label{mub}
F(t) \geq 1
\end{equation}
for all $t \in \R/\Z$.

On the other hand, from \eqref{mvd} we have the first moment
\begin{equation}\label{mu1}
 \int_{\R/\Z} F(t)\ dt = 2 n \delta_n = 1 + \frac{2A}{n}
\end{equation}
and hence $F-1$ has a small integral:
\begin{equation}\label{mu0}
 \int_{\R/\Z} (F(t)-1)\ dt = \frac{2A}{n}.
\end{equation}
Now we lower bound the second moment. 

\begin{lemma}\label{ark}  We have
\begin{equation}\label{mu3}
 \int_{\R/\Z} F(t)^2\ dt \geq 1+c - o(1)
\end{equation}
for some absolute constant $c>0$.
\end{lemma}

\begin{proof}
The left-hand side can be expanded as
$$ \sum_{i=1}^n \sum_{j=1}^n m( B(v_i,v_j;\delta_n,\delta_n) ),$$
where $m$ denotes Lebesgue measure.
By \eqref{deltan}, the contribution of the diagonal case $i=j$ is $2n\delta_n = 1 + o(1)$, thanks to \eqref{dnd-2}.  Thus, to complete the proof of \eqref{mu3}, it will suffice to establish the lower bound
\begin{equation}\label{lb}
 m( B(v_i,v_j;\delta_n,\delta_n) ) \gg \frac{1}{n^2}
\end{equation}
for all $1 \leq i,j \leq n$.  But this follows from Corollary \ref{clb} and \eqref{dnd-2}.
\end{proof}

\begin{remark}  Using \cite[Proposition 2]{ps}, one can in fact take $c = \frac{1}{2}$, basically because \cite[Corollary 9]{ps} allows one to take the implied constant in \eqref{lb} to be $2$.  However, the exact value of $c$ will not be important for our argument, so long as it is positive.
\end{remark}

From \eqref{mu3} and \eqref{mu1} we conclude in particular that
$$ \int_{\R/\Z} (F(t)-1)^2\ dt = \int_{\R/\Z} F(t)^2\ dt - 1 - o(1) \geq c - o(1)$$
so for $n$ large enough we have
$$ 1 \leq \int_{\R/\Z} F(t)^2\ dt \ll \int_{\R/\Z} (F(t)-1)^2\ dt.$$
From \eqref{mub}, \eqref{mu0}, H\"older's inequality (or Cauchy-Schwarz), and \eqref{mu3}, we thus have the third moment bound
\begin{align*}
 \int_{\R/\Z} F(t)^3\ dt &\geq  \int_{\R/\Z} (F(t)-1)^3\ dt \\
&\geq \frac{\left(\int_{\R/\Z} (F(t)-1)^2\ dt\right)^2}{\int_{\R/\Z} (F(t)-1)\ dt} \\
&= \frac{n}{2A} \left(\int_{\R/\Z} (F(t)-1)^2\ dt\right)^2 \\
&\gg \frac{n}{A} \left(\int_{\R/\Z} F(t)^2\ dt\right)^2 \\
&\gg \frac{n}{A} \int_{\R/\Z} F(t)^2\ dt.
\end{align*}

Expanding out the second and third moments of $F$, we conclude that
$$ \sum_{1 \leq i,j,k \leq n} m( B(v_i,v_j,v_k;\delta_n,\delta_n,\delta_n) )
\gg \frac{n}{A} \sum_{1 \leq i,j \leq n} m( B(v_i,v_j; \delta_n,\delta_n )  )
$$
and hence by the pigeonhole principle, there exist $1 \leq i,j \leq n$ (not necessarily distinct) such that
$$ \sum_{1 \leq k \leq n} m( B(v_i,v_j,v_k;\delta_n,\delta_n,\delta_n) )
\gg \frac{n}{A} m( B(v_i,v_j; \delta_n, \delta_n) ).
$$
Henceforth $i,j$ are fixed.  Clearly, the majority of the contribution to the sum on the left-hand side will come from those $k$ for which
\begin{equation}\label{mamba}
 m( B(v_i,v_j,v_k;\delta_n,\delta_n,\delta_n) ) \gg
\frac{1}{A} m( B(v_i,v_j; \delta_n, \delta_n) ).
\end{equation}
On the other hand, we have the trivial upper bound
$$ m( B(v_i,v_j,v_k;\delta_n,\delta_n,\delta_n) )  \leq m( B(v_i,v_j;\delta_n,\delta_n) )$$
(note this already recovers the bound \eqref{a-large}, which is of comparable strength to the existing bounds \eqref{dd}, \eqref{ee}). 
Subdividing into $O(\log(1+A)) = O(\log \log n)$ dyadic regions and using the pigeonhole principle, we thus conclude the existence of some
\begin{equation}\label{a1-card}
1 \ll A_1 \ll A
\end{equation}
such that
$$ \sum_{1 \leq k \leq n: \frac{n}{2A_1} \leq m( B(v_i,v_j,v_k;\delta_n,\delta_n,\delta_n) ) \leq \frac{n}{A_1} } m( B(v_i,v_j,v_k;\delta_n,\delta_n,\delta_n) )
\gg \frac{n}{A \log\log n} m( B(v_i,v_j; \delta_n, \delta_n) ).
$$
In particular, we see that the estimate
$$
 m( B(v_i,v_j,v_k; \delta_n,\delta_n,\delta_n) ) \asymp \frac{1}{A_1} m( B(v_i,v_j; \delta_n, \delta_n) ) $$
holds for $\gg \frac{A_1 n}{A \log\log n}$ values of $k=1,\dots,n$.

Henceforth $A_1$ is fixed.  Applying Lemma \ref{lod} and \eqref{dnd-2}, we conclude that the estimate
\begin{equation}\label{ag}
\mu\left( P\left(v_i,v_j,v_k;\frac{1}{\delta_n},\frac{1}{\delta_n},\frac{1}{\delta_n}\right) \right)
\asymp \frac{n}{A_1} \mu\left( P\left( v_i,v_j; \frac{1}{\delta_n},\frac{1}{\delta_n}\right) \right)
\end{equation}
holds for $\gg \frac{A_1 n}{A \log\log n}$ values of $k=1,\dots,n$.

Informally, the estimate \eqref{ag} is asserting that many velocities of the $v_k$ are somehow arithmetically related to the fixed velocities $v_i, v_j$.  We make this precise as follows.  

\begin{proposition}\label{new} Let $k = 1,\dots,n$ obey \eqref{ag}.  Then there exists a positive integer 
\begin{equation}\label{ak}
a_k = O(A_1)
\end{equation}
such that we have a linear relation between $v_i,v_j,v_k$ of the form
\begin{equation}\label{avk}
 a_k v_k = n_{i,k} v_i + n_{j,k} v_j
\end{equation}
with $|n_{i,k}|, |n_{j,k}| \leq \frac{2}{\delta_n}$.  In particular, from \eqref{ak}, \eqref{asm}, \eqref{a1-card} we have
\begin{equation}\label{ak-small}
a_k = o( \log n ).
\end{equation}
\end{proposition}

\begin{proof}
It will be convenient to write the fraction $\frac{v_i}{v_j}$ in lowest terms as 
\begin{equation}\label{vhh}
\frac{v_i}{v_j} = \frac{h_i}{h_j}
\end{equation}
for some coprime non-zero integers $h_i,h_j$.  Set $H \coloneqq \max(|h_i|, |h_j|)$ to be the height of this fraction.  

The equation $n_i v_i + n_j v_j = 0$ is only solvable in integers when $(n_i,n_j)$ is an integer multiple of $(h_j,-h_i)$, which is a vector of magnitude $\sim H$.  From \eqref{dnd-2} and the definition of multiplicity, we conclude that
$$ \mu\left( P\left(v_i, v_j; \frac{1}{\delta_n}, \frac{1}{\delta_n} \right)\right) \asymp 1 + \frac{n}{H}.$$
For similar reasons, we see that every integer has at most $O(1 + \frac{n}{H})$ representations of the form $n_i v_i + n_j v_j$ with $|n_i|, |n_j| \leq \frac{1}{\delta_n}$, and hence
$$ \mu\left( P\left(v_i,v_j,v_k;\frac{1}{\delta_n},\frac{1}{\delta_n},\frac{1}{\delta_n} \right)\right) \ll \left(1 + \frac{n}{H}\right) \#\left(P_{ij} \cap P_k\right)$$
where $P_{ij}$ is the rank two generalised arithmetic progression
$$ P_{ij} \coloneqq P\left(v_i,v_j; \frac{1}{\delta_n},\frac{1}{\delta_n}\right)$$
and $P_k$ is the rank one progression
$$ P_k \coloneqq P\left( v_k; \frac{1}{\delta_n} \right).$$
From \eqref{ag}, we thus have
\begin{equation}\label{ncp}
 \#(P_{ij} \cap P_k) \gg \frac{n}{A_1}
\end{equation}
for $\gg \frac{A_1 n}{A \log\log n}$ values of $k=1,\dots,n$.

Let $k=1,\dots,n$ obey \eqref{ncp}.  Every element of $P_{ij} \cap P_k$ clearly lies in $P_k$, and is thus of the form $n_k v_k$ for some $n_k = O(A)$.  On the other hand, from \eqref{ncp} there are $\gg \frac{n}{A_1}$ such elements $n_k v_k$.  By the pigeonhole principle (or Dirichlet box principle), there must therefore exist distinct elements $n_k v_k, n'_k v_k$ of the set $P_{ij} \cap P_k$ with $n'_k - n_k = O(A_1)$.  Subtracting, we conclude that that there exists a positive integer $a_k$ of size $O(A_1)$ such that $a_k v_k \in 2P_{ij}$, and the claim follows.
\end{proof}

As in the proof of the above lemma, we write $v_i/v_j$ in lowest terms using \eqref{vhh}, thus we may write $v_i = h_i v_0$ and $v_j = h_j v_0$ for some non-zero integer $v_0$.  We have already seen that $v_k$ is arithmetically related to $v_i$ and $v_j$.  We now show that $v_k$ is also arithmetically related to $v_0$:

\begin{proposition}\label{will}  After removing at most $n^{3/4+o(1)}$ exceptional choices of $k$, for all remaining $k$ obeying \eqref{ag}, there exists a positive integer
\begin{equation}\label{apk}
1 \leq a'_k \ll A_1
\end{equation}
and an integer $n'_{0,k}$ coprime to $a'_k$ with
\begin{equation}\label{ank}
|n'_{0,k}| \ll A_1 n
\end{equation}
such that
\begin{equation}\label{avp}
 a'_k v_k = n'_{0,k} v_0. 
\end{equation}
\end{proposition}

\begin{proof}
We first dispose of a degenerate case in which $n_{i,k}, n_{j,k}$ are both small, say $|n_{i,k}|, |n_{j,k}| \leq n^{1/3}$.  The number of triples $(a_k, n_{i,k},n_{j,k})$ of this form does not exceed
$$ O\left( A_1 \times n^{1/3} \times n^{1/3} \right) = O(n^{2/3+o(1)}).$$
Since the $v_k$ are all distinct, we conclude from \eqref{avk} that there are at most $O(n^{2/3+o(1)})$ values of $k$ which are of this form.  Discarding these $k$ as exceptional, we may assume that
\begin{equation}\label{nnn}
 \max( |n_{i,k}|, |n_{j,k}| ) > n^{1/3}.
\end{equation}
In particular, if we let $m_k$ be the largest natural number such that $m_k |n_{i,k}|, m_k |n_{j,k}| \leq \frac{2}{\delta_n}$, then we have
\begin{equation}\label{mk}
 1 \leq m_k \ll n^{2/3}.
\end{equation}

Next, we recall the progressions $P_{ij}, P_k$ from the proof of Proposition \ref{new}.  Since $v_i = h_i v_0$ and $v_j = h_j v_0$ with $H \coloneqq \max(|h_i|, |h_j|)$, we see that 
\begin{equation}\label{p-inc}
P_{ij} \subset P( v_0; 10 H n )
\end{equation}
and similarly
\begin{equation}\label{2p-inc}
2P_{ij} \subset P( v_0; 20 H n ).
\end{equation}
Since $a_k v_k$ lies in $2P_{ij}$, we conclude that
$$ a_k v_k = n_{0,k} v_0 $$
for some integer $n_{0,k}$ with $|n_{0,k}| \leq 20 Hn$.  Reducing to lowest terms, we thus have \eqref{avp} for some positive integer $a'_k$ obeying \eqref{apk}  and some integer $n'_{0,k}$ coprime to $a'_k$ with 
\begin{equation}\label{nku}
|n'_{0,k}| \leq 20 Hn.
\end{equation}
If $|n'_{0,k}| \leq n^{3/4}$, then the number of possible pairs $(a'_k, n'_{0,k})$ is at most $O(n^{3/4+o(1)})$, so we may again discard these $k$ as exceptional.  Thus we may assume that
\begin{equation}\label{nk}
|n'_{0,k}| > n^{3/4}.
\end{equation}

We now divide into three cases depending on the size of the height $H$.  Let us first suppose that we are in the \emph{highly incommensurable} case when $H \geq n \log^{10} n$; this informally corresponds to the case where $P_{ij}$ behaves ``two-dimensionally''.  Then we have
 $m a_k v_k \not \in 2P_{ij}$ whenever $m_k < m < \frac{\log^{10} n}{10} m_k$ (say).  This implies that for any integer $v$, the set 
$$ \left\{ m \in \Z: 0 \leq m < \frac{\log^{10} n}{10} m_k: v + m a_k v_k \in P_{ij} \right\}$$
can have cardinality at most $m_k$ (indeed this set must have diameter at most $m_k$).  As a consequence, the set $P_{ij} \cap P_k$ appearing in \eqref{ncp} intersects each arithmetic progression of the form $\{ (m_0 + ma_k) v_k: m \in \Z: 0 \leq m < \frac{\log^{10} n}{10} m_k \}$ in a set of cardinality at most $m_k$ (so in particular in a set of relative density $O(\log^{-10} n)$ in the arithmetic progression).  Covering $P_k$ by $O\left( \frac{\frac{1}{\delta_n}}{\frac{\log^{10} n}{10} m_k}\right)$ such progressions (noting from \eqref{mk}, \eqref{ak-small}, \eqref{dnd-2} that $a_k \frac{\log^{10} n}{10} m_k$ is significantly smaller than $\frac{1}{\delta_n}$), we conclude that 
$$ P_{ij} \cap P_k \ll n \log^{-10} n.$$
But this, using \eqref{asm}, contradicts \eqref{ncp} if $n$ is large enough.  Thus the highly incommensurable case does not occur.

Now we consider the \emph{commensurable} case when $H \leq 10n$; this informally corresponds to the case where $P_{ij}$ behaves ``one-dimensionally''.   Using \eqref{p-inc} and \eqref{avp}, we obtain the inclusion
$$
P_{ij} \cap P_k \subset P( v_0; 10 H n ) \cap P_k \subset P\left( a'_k v_k; \frac{10 Hn}{|n'_{0,k}|} \right).$$
In particular we have
$$ \#(P_{ij} \cap P_k) \ll 1 + \frac{10 Hn}{|n'_{0,k}|}.$$
Comparing this with \eqref{ncp} and using $H \leq 10n$ (and \eqref{a1-card}, \eqref{asm}), we obtain \eqref{ank} as desired.

Finally, we consider the \emph{moderately incommensurable case} $10n < B < n \log^{10} n$.  We treat this case by a combination of the two preceding arguments.  As in the commensurable case, we have the inclusion
$$
P_{ij} \cap P_k \subset P\left( a'_k v_k; \frac{10 Hn}{|n'_{0,k}|} \right).$$
On the other hand, by repeating the highly incommensurable arguments, we see that the set $P_{ij} \cap P_k$ intersects each arithmetic progression of the form 
\begin{equation}\label{form}
\left\{ (m_0 + ma_k) v_k: m \in \Z: 0 \leq m < \frac{H}{10 n} m_k \right\}
\end{equation}
in a set of relative density $O(\frac{n}{H} )$ in those progressions.
Recall that $a_k$ is a multiple of $a'_k$, and by \eqref{nku}, \eqref{mk}, and \eqref{ak-small} we have
$$ a_k \frac{H}{10 n} m_k \ll a'_k \frac{10 Hn}{|n'_{0,k}|}.$$
Thus we may cover $P( a'_k v_k; \frac{10 Hn}{|n'_{0,k}|} )$ by $O\left( \frac{\frac{10 Hn}{|n'_{0,k}|}}{\frac{H}{10 n} m_k} \right)$ progressions of the form \eqref{form} and conclude that
$$ \#(P_{ij} \cap P_k) \ll  \frac{n}{H} \frac{Hn}{ |n'_{0,k}|} \ll \frac{n^2}{|n'_{0,k}|}.$$
Comparing this with \eqref{ncp} we again conclude \eqref{ank} as desired.
\end{proof}

From the above proposition (and \eqref{asm}, \eqref{a1-card}), we see that all $k$ in a subset $K \subset \{1,\dots,n\}$ of cardinality
$$
\# K \gg \frac{A_1 n}{A \log\log n},
$$
one can write
\begin{equation}\label{vo}
 v_k = n'_{0,k} \frac{v_0}{a'_k}
\end{equation}
for some positive integer 
\begin{equation}\label{appk}
a'_k = o( \log n )
\end{equation}
and some integer 
\begin{equation}\label{nok}
n'_{0,k} = O( A_1 n ).
\end{equation}

We partition $K$ into $K_1 \cup K_2$, where $K_1$ is the set of those $k \in K$ for which $n'_{0,k} = n'_{0,l}$ for some $l \in K \backslash \{k\}$, and $K_2$ is the remaining set of $k \in K$.  By the pigeonhole principle we have
\begin{equation}\label{k-card}
\# K_a \gg \frac{A_1 n}{A \log\log n}
\end{equation}
for some $a=1,2$, which we now fix.

Now we are ready to improve the union bound.  From \eqref{mub} we have the pointwise inequality
$$ F(t) - 1 \geq \frac{1}{2} F(t) 1_{F(t) \geq 2},$$
so upon integrating and using \eqref{mu0} we have
\begin{equation}\label{rod}
\frac{2A}{n} \geq \frac{1}{2} \int_{\R/\Z} F(t) 1_{F(t) \geq 2}\ dt.
\end{equation}
We can expand
\begin{equation}\label{dor}
\begin{split}
\int_{\R/\Z} F(t) 1_{F(t) \geq 2}\ dt &= \sum_{k=1}^n m( B(v_k;\delta_n) \cap \{ F \geq 2 \} ) \\
&= \sum_{k=1}^n m\left( B(v_k;\delta_n) \cap \bigcup_{1 \leq l \leq n: l \neq k} B(v_l;\delta_n) \right) \\
&\geq \sum_{k \in K_a} m\left( B(v_k;\delta_n) \cap \bigcup_{l \in K_a: l \neq k} B(v_l;\delta_n)\right ) 
\end{split}
\end{equation}
and thus
\begin{equation}\label{agn}
 A \gg n \sum_{k \in K_a} m\left( B(v_k;\delta_n) \cap \bigcup_{l \in K_a: l \neq k} B(v_l;\delta_n) \right).
\end{equation}

We now divide into two cases, depending on the value of $a$.  First suppose that $a=1$.  Then for every $k \in K_a$, there is $l \in K_a$ distinct from $k$ such that $n'_{0,k} = n'_{0,l}$.  From \eqref{vo} we then have the inclusion
$$ B(v_k,v_l; \delta_n,\delta_n) \supset B\left( \frac{n'_{0,k} v_0}{b_{k,l}}; \frac{\delta_n}{a'_k a'_l} \right)$$
where $b_{k,l}$ is the least common multiple of $a'_k$ and $a'_l$, so in particular (by Corollary \ref{clb} or \eqref{mvd} and  \eqref{dnd-2}, \eqref{appk})
$$ m\left( B(v_k;\delta_n) \cap \bigcup_{l \in K_a: l \neq k} B(v_l;\delta_n) \right) \gg \frac{1}{n \log^2 n}.$$
Inserting this into \eqref{agn} and using \eqref{k-card} we conclude that
$$ A \gg \frac{A_1 n}{A \log^2 n\log\log n},$$
which contradicts \eqref{asm}, \eqref{a1-card} (with substantial room to spare).

Now suppose that $a=2$, then the $n'_{0,k}$ are all distinct as $k$ varies in $K_a$.

Define a \emph{medium-sized prime} to be a prime $p$ in the range $\log^{10} n \leq p \leq n^{1/10}$.  Suppose that $p$ is a medium-sized prime that divides both $n'_{0,k}$ and $n'_{0,l}$ for some distinct $k,l \in K_a$.  Then the frequencies $v_k = n'_{0,k} \frac{v_0}{a'_k}$ and $v_l = n'_{0,l} \frac{v_0}{a'_l}$ are both integer multiples of $\frac{pv_0}{b_{k,l}}$, where $b_{k,l} = o(\log^2 n)$ is the greatest common divisor of $v_0$ and $a'_k a'_l$ (note from \eqref{appk} that the medium-sized prime $p$ will not divide $a'_k$ or $a'_l$, while from \eqref{avp} we know that $a'_k$ divides $v_0$).  From \eqref{nok}, \eqref{dnd-2}, \eqref{appk}, \eqref{asm} we then have the inclusion
$$ B(v_k,v_l; \delta_n, \delta_n) \supset B\left( \frac{p v_0}{b_{k,l}}, \frac{p}{n^2 \log^5 n} \right)$$
(say).  In particular we have
$$ B(v_k,v_l; \delta_n, \delta_n) \supset B_{p, b_{k,l}} $$
where $B_{p,b_{k,l}}$ is the ``major arc'' set 
$$ B_{p,b_{k,l}} \coloneqq \bigcup_{c=1}^{p-1} \left\{ t \in \R/\Z: \left\| \frac{tv_0}{b_{k,l}} - \frac{c}{p} \right\|_{\R/\Z} \leq \frac{1}{n^2 \log^5 n} \right\}.$$
Observe that each set $B_{p,b_{k,l}}$ has measure
$$ m(B_{p,b_{k,l}}) = \frac{2(p-1)}{n^2 \log^5 n} \gg \frac{\log^5 n}{n^2}.$$
Also, we claim that if $p,p'$ are two distinct medium-sized primes, and $b,b' = o(\log^2 n)$ are positive integers dividing $v_0$, then the sets $B_{p,b}$ and $B_{p',b'}$ are disjoint.  This will be a variant of the arguments in the introduction.  Suppose for contradiction that there was $t \in \R/\Z$ that was in both $B_{p,b}$ and $B_{p',b'}$, then we have
$$ \left\| \frac{tv_0}{b} - \frac{c}{p} \right\|_{\R/\Z}, \left\| \frac{tv_0}{b'} - \frac{c'}{p'} \right\|_{\R/\Z} \leq \frac{1}{n^2 \log^5 n}
$$
for some $1 \leq c \leq p-1$ and $1 \leq c' \leq p'-1$.  Eliminating $t$ using the triangle inequality, we conclude that
\begin{equation}\label{cont}
 \left\| \frac{bc}{p} - \frac{b'c'}{p'} \right\|_{\R/\Z}\leq \frac{b+b'}{n^2 \log^5 n}
\end{equation}
As $p,p'$ are medium-sized primes, they do not divide $b,b' = o(\log^2 n)$, and so the fraction $\frac{bc}{p} - \frac{b'c'}{p'}$ is non-integer.  In particular
$$ \left\| \frac{bc}{p} - \frac{b'c'}{p'} \right\|_{\R/\Z}\geq \frac{1}{pp'} \geq \frac{1}{n^{1/5}},$$
which contradicts \eqref{cont} since $b,b' = o(\log^2 n)$.

In view of the above facts, we see that if a medium-sized prime $p$ divides $n'_{0,k}$ for $r$ values of $k \in K_a$, and $r \geq 2$, then this prime contributes $\gg r \frac{\log^5 n}{n^2}$ to the sum in \eqref{agn}; furthermore, the contributions of different medium-sized primes are disjoint.  We conclude that
$$ A \gg n \sum_{p} \frac{\log^5 n}{n^2} \left( \# \{ k \in K_a: p | n_{0,k} \} - 1 \right)$$
where in the remainder of the argument, $p$ is understood to range over medium-sized primes.
The contribution of the $-1$ term can be crudely bounded by $O( \sum_{p} \frac{\log^5 n}{n}) =o(1)$, hence
\begin{align*}
 A &\gg \sum_{p} \frac{\log^{5} n}{n} \# \{ k \in K_a: p | n_{0,k} \} - o(1) \\
&= \frac{\log^{5} n}{n} \sum_{k \in K_a} \sum_{p: p | n_{0,k}} 1 - o(1).
\end{align*}

Standard sieve bounds (see e.g. \cite[Corollary 6.2]{ik}) show that the number of integers in the set $\{ n' \in \Z: n' = O( A_1 n ) \}$ which are not divisible by any medium-sized prime $p$ is at most
$$ O\left( \frac{\log(\log^{10} n)}{\log(n^{1/10})} A_1 n\right) = O\left( \frac{\log\log n}{\log n} A_1 n \right).$$
Since the $n'_{0,k}$ for $k \in K_a$ are distinct, this implies that the number of $k \in K_a$ with $n'_{0,k}$ not divisible by any medium-sized prime is also $O( \frac{\log\log n}{\log n} A_1 n )$. Comparing this with \eqref{k-card} and \eqref{asm}, we see that there are $\gg \frac{A_1 n}{A \log\log n}$ elements $k$ of $K_a$ that are divisible by at least one medium-sized prime $p$.  We conclude that
$$ A \gg \frac{\log^{5} n}{n} \frac{A_1 n}{A \log\log n} - o(1)$$
which contradicts \eqref{asm}, \eqref{a1-card} for $n$ large enough (with some room to spare).  The claim follows.

\section{Proof of second theorem}\label{second-sec}

We now prove Theorem \ref{main2}.  Let $C_1$ be a large constant to be chosen later, and set $C_0 \coloneqq C_1^2$.  We prove the theorem by induction on $n_0$.  The claim is trivial for $n_0=1$, so suppose that $n_0 > 1$ and that the claim has already been proven for smaller values of $n_0$.  The implication of (ii) from (i) is trivial, so it remains to assume (ii) and establish (i).  By the induction hypothesis, we already have
\begin{equation}\label{delta-induct}
 \delta_n = \frac{1}{n+1}
\end{equation}
for all $n < n_0$, and (in view of \eqref{deltan}) our task is then to show that
\begin{equation}\label{jip}
 \delta(v_1,\dots,v_{n_0}) \geq \frac{1}{n_0+1}
\end{equation}
for any non-zero distinct integers $v_1,\dots,v_{n_0}$.  From \eqref{delta-induct}, we see that
\begin{equation}\label{pij}
 \delta(v'_1,\dots,v'_{n_0}) \geq \frac{1}{n_0}
\end{equation}
whenever $v'_1,\dots,v'_{n_0}$ are non-zero integers with at least one collision $v'_i = v'_j$ for $1 \leq i < j \leq n_0$, since one can remove all duplicate velocities and apply \eqref{delta-induct} with the surviving $n$ velocities for some $n < n_0$.  As mentioned in the introduction, the strategy is to compare $\delta(v_1,\dots,v_{n_0})$ to some $\delta(v'_1,\dots,v'_{n_0})$ involving a collision if the $v_1,\dots,v_{n_0}$ are not already efficiently contained in a (rank one) arithmetic progression.

We turn to the details.  For brevity we now abbreviate $n_0$ as $n$ henceforth. Let $v_1,\dots,v_{n}$ be non-zero distinct integers.  Applying Proposition \ref{ppp} with $P = P(v_1,\dots,v_{n}; 1,\dots,1)$ and $t = n^{C_1 n}$, we can find a $n^{C_1 n}$-proper generalised arithmetic progression $Q = Q(w_1,\dots,w_r; N_1,\dots,N_r)$ of rank $r \leq n$ that contains all of the $v_1,\dots,v_{n}$, with
$$ \# Q \ll n^{O(C_1 n^2)}.$$
Let $\phi \colon \R^r \to \R$ denote the linear map
$$\phi(n_1,\dots,n_r) \coloneqq n_1 w_1 + \dots + n_r w_r,$$
then by the construction of $Q$ we have
\begin{equation}\label{vpa}
v_i = \phi(a_i)
\end{equation}
for $i=1,\dots,n$ and some $a_i \in \Z^r$ that lie in the box 
$$\{ (n_1,\dots,n_r) \in \R^r: |n_1| \leq N_1,\dots,|n_r| \leq N_r \}.$$

We now need an elementary lemma that allows us to create a ``collision'' between two of the $a_1,\dots,a_{n}$ via a linear projection, without making any of the $a_i$ collide with the origin:

\begin{lemma}  Let $a_1,\dots,a_{n} \in \R^r$ be non-zero vectors that are not all collinear with the origin.  Then, after replacing one or more of the $a_i$ with their negatives $-a_i$ if necessary, there exists a pair $a_i,a_j$ such that $a_i-a_j \neq 0$, and such that none of the $a_1,\dots,a_{n}$ is a scalar multiple of $a_i-a_j$.
\end{lemma}

\begin{proof}  We may assume that $r \geq 2$, since the $r \leq 1$ case is vacuous.  Applying a generic linear projection to $\R^2$ (which does not affect collinearity, or the property that a given $a_k$ is a scalar multiple of $a_i-a_j$), we may then reduce to the case $r=2$.  

By a rotation and relabeling, we may assume that $a_1$ lies on the negative $x$-axis; by flipping signs as necessary we may then assume that all of the $a_2,\dots,a_n$ lie in the closed right half-plane.  As the $a_i$ are not all collinear with the origin, one of the $a_i$ lies off of the $x$-axis, by relabeling, we may assume that $a_2$ lies off of the $x$-axis and makes a minimal angle with the $x$-axis.  Then the angle of $a_2-a_1$ with the $x$-axis is non-zero but smaller than any non-zero angle that any of the $a_i$ make with this axis, and so none of the $a_i$ are a scalar multiple of $a_2-a_1$, and the claim follows.
\end{proof}

We now return to the proof of the proposition.  If the $a_1,\dots,a_{n}$ are all collinear with the origin, then $\phi(a_1),\dots,\phi(a_{n})$ lie in a one-dimensional arithmetic progression $P( v; \# Q)$; by rescaling we may then take $v_1,\dots,v_{n}$ to be integers of magnitude at most $\# Q \ll n^{O(C_1 n^2)}$, and the claim \eqref{jip} then follows from the hypothesis (ii) if $C_1$ is large enough, since $C_0 = C_1^2$.  Thus, we may assume that the $a_1,\dots,a_{n}$ are not all collinear with the origin, and so by the above lemma and relabeling we may assume that $a_{n}-a_1$ is non-zero, and that none of the $a_i$ are scalar multiples of $a_{n}-a_1$.

We will replace the velocities $v_i = \phi(a_i)$ by a variant $v'_i = \phi'(a_i)$, where $\phi' \colon \R^r\to \R$ is a modification of $\phi \colon \R^r \to \R$ designed to create a collision.  To construct $\phi'$, we write
\begin{equation}\label{aa}
 a_{n}-a_1 = q \vec c
\end{equation}
where $q$ is a positive integer and $\vec c = (c_1,\dots,c_r) \in \Z^r$ is a vector whose coefficients $c_j$ have no common factor and obey the bound$|c_j| \leq 2 N_i$ for $j=1,\dots,r$; by relabeling we may assume without loss of generality that $c_r$ is non-zero, and furthermore that
\begin{equation}\label{cjrp}
 \frac{|c_j|}{N_j} \leq \frac{|c_r|}{N_r}
\end{equation}
for $j=1,\dots,r$.  

We now define a variant $\phi' \colon \R^r \to \R$ of $\phi \colon \R^r \to \R$ by the formula
$$ \phi'(n_1,\dots,n_r) \coloneqq \sum_{j=1}^r w'_j (n_j c_r - n_r c_j),$$
where the $w'_1,\dots,w'_r$ are an ``extremely lacunary'' sequence of  integers; the precise form of $w'_j$ is not important for our argument, but for sake of concreteness we set
$$ w'_j \coloneqq n^{j C_1 n^{100}}.$$
We then set $v'_1,\dots,v'_{n}$ to be the integers 
\begin{equation}\label{vpi}
 v'_i \coloneqq \phi'(a_i).
\end{equation}
By contruction, the map $\phi' \colon \R^r \to \R$ is linear and annihilates $a_{n}-a_1$, hence we have a collision
$$ v'_1 = v'_{n}.$$
We also have the non-vanishing of the $v'_i$:

\begin{lemma}\label{none-vanish}  One has $v'_i \neq 0$ for every $i=1,\dots,n$.
\end{lemma}

\begin{proof} If we write $a_i = (n_{1,i}, \dots, n_{r,i})$, then by construction $a_i$ is not parallel to $a_{n}-a_1$, and is thus not a multiple of $\vec c$.  In particular, at least one of the coefficients $n_{j,i} c_r - n_{r,i} c_j$ of the quantity
$$ v'_i = \phi'(a_i) = \sum_{j=1}^r w'_j (n_{j,i} c_r - n_{r,i} c_j)$$
is non-vanishing.  On the other hand, these coefficients are integers of size 
$$ O( N_j N_r ) = O( (\# Q)^2 ) = n^{O(C_1 n^2)}.$$
Given the highly lacunary nature of the $w'_j$, we conclude the non-vanishing of $v'_i$ as claimed.
\end{proof}

Applying \eqref{pij}, we conclude that
$$ \delta(v'_1,\dots,v'_{n}) \geq \frac{1}{n}.$$
We now use Fourier-analytic techniques to ``transfer'' this bound to obtain \eqref{jip}.  By definition, there exists $t_0 \in \R/\Z$ such that
$$ \min_{1 \leq i \leq n} \| t_0 v'_i \|_{\R/\Z} \geq \frac{1}{n}.$$
On the other hand, from Corollary \ref{clb}, the set of $t \in \R/\Z$ for which
$$ \max_{1 \leq i \leq n} \| t v'_i \|_{\R/\Z} \leq \frac{1}{10 n^2}$$
has measure $\gg n^{-O(n)}$.  By the triangle inequality, and shifting the above set by $t_0$, we conclude that
\begin{equation}\label{nido}
 \min_{1 \leq i \leq n} \| t v'_i \|_{\R/\Z} \geq \frac{1}{n} - \frac{1}{10 n^2}
\end{equation}
for all $t$ in a subset of $\R/\Z$ of measure $\gg n^{-O(n)}$.

We now need a certain smooth approximant to the indicator function $1_{\|x\|_{\R/\Z} \leq \frac{1}{n+1}}$.

\begin{lemma}\label{lam} There exists a trigonometric polynomial $\eta \colon \R/\Z \to \R$ of the form
\begin{equation}\label{ham}
 \eta(x) = \sum_{m: |m| \leq n^{C_1 n/10}} b_m e^{2\pi i mx} 
\end{equation}
for some complex coefficients $b_m$, which takes values in $[0,1]$ and is such that
\begin{equation}\label{nots}
 \eta(x) \gg 1
\end{equation}
when $\|x\|_{\R/\Z} \geq \frac{1}{n} - \frac{1}{10n^2}$ and
\begin{equation}\label{ston}
 \eta(x) \ll n^{-100 C_1 n} 
\end{equation}
when $\|x\|_{\R/\Z} \geq \frac{1}{n+1}$.
\end{lemma}

\begin{proof}  The function $1_{\|x\|_{\R/\Z} \geq \frac{1}{n} - \frac{1}{5n^2}}$ has a Fourier expansion
$$ 1_{\|x\|_{\R/\Z} \geq \frac{1}{n} - \frac{1}{5n^2}} = \sum_{m \in \Z} a_m e^{2\pi i m x}$$
for some square-summable coefficients $a_m$ (where the series convergence is in the $L^2$ sense).    Let $\phi \colon \R \to \R$ be a Schwartz class non-negative function supported on $[-1,1]$ with positive Fourier transform $\hat \phi$; we may normalise $\phi(0)=1$, so that $\int_\R \hat \phi(s)\ ds = 1$.  We define $\eta$ to be the trigonometric polynomial
$$ \eta(x) \coloneqq \sum_{m: |m| \leq n^{C_1 n/10}} \phi\left( \frac{m}{n^{C_1 n/10}}\right ) a_m e^{2\pi i mx}.$$
From the Fourier inversion formula, we can also write $\eta$ as a convolution:
$$ \eta(x) = \int_\R \hat \phi(s) 1_{\|x - \frac{s}{n^{C_1 n/10}} \|_{\R/\Z} \geq \frac{1}{n} - \frac{1}{5n^2}}\ ds.$$
Since $\hat \phi$ is non-negative and has total mass $1$, we now see that $\eta$ takes values in $[0,1]$ as claimed.  If $\|x\|_{\R/\Z} \geq \frac{1}{n+1}$, then the constraint $\|x - \frac{s}{n^{C_1 n/10}} \|_{\R/\Z} \geq \frac{1}{n} - \frac{1}{5n^2}$ can only be satisfied if $s$ is larger than $n^{C_1 n/20}$ (say), so the claim \eqref{ston} follows from the rapid decrease of $\hat \phi$.  Finally, if 
$\|x\|_{\R/\Z} \geq \frac{1}{n} - \frac{1}{10n^2}$, then the constraint $\|x - \frac{s}{n^{C_1 n/10}} \|_{\R/\Z} \geq \frac{1}{n} - \frac{1}{5n^2}$ is obeyed for all $s \in [-1,1]$, giving \eqref{nots}.
\end{proof}

Let $\eta$ be as in the above lemma.  From \eqref{nido}, we have
$$ \eta( t v'_i ) \gg 1$$
for all $i=1,\dots,n$ and all $t$ in a subset of $\R/\Z$ of measure $\gg n^{-O(n)}$.   Multiplying and integrating, we conclude that
\begin{equation}\label{ideo}
 \int_{\R/\Z} \prod_{i=1}^{n} \eta(t v'_i)\ dt \gg n^{-O(n)}.
\end{equation}

Now we come to the key Fourier-analytic comparison identity.

\begin{lemma}[Comparison identity]  Let $D \colon \R/\Z \to \C$ denote the Dirichlet series
\begin{equation}\label{vphi-def}
 D(x) \coloneqq \sum_{m: |m| \leq n^{C_1 n/2} \frac{N_r}{|c'_r|}} e^{2\pi i mx}.
\end{equation}
Then
\begin{equation}\label{lahs}
 \int_{\R/\Z} \prod_{i=1}^{n} \eta(t v'_i)\ dt = \int_{\R/\Z} D(t \phi(\vec c)) \prod_{i=1}^{n} \eta(t v_i)\ dt.
\end{equation}
\end{lemma}

\begin{proof}
Using \eqref{ham}, \eqref{vpi}, and the linearity of $\phi'$, the left-hand side of \eqref{lahs} may be expanded as
$$ \sum_{m_1,\dots,m_{n} \in \Z: \phi'( m_1 a_1 + \dots + m_n a_n) = 0} b_{m_1} \dots b_{m_n}$$
where we adopt the convention that $b_m$ vanishes when $m > n^{C_1 n/10}$.  Similarly, by \eqref{ham}, \eqref{vpa}, \eqref{vphi-def}, the right-hand side may be expanded as
$$ \sum_{|m| \leq n^{C_1 n/2} \frac{N_r}{|c_r|}} \sum_{m_1,\dots,m_{n} \in \Z: \phi(m_1 a_1 + \dots + m_n a_n + m \vec c) = 0} b_{m_1} \dots b_{m_n}.$$
Thus, to prove \eqref{lahs}, it suffices to show that for any integers $m_1,\dots,m_{n}$ with $|m_i| \leq n^{C_1 n/10}$, one has
\begin{equation}\label{chy}
\phi'( m_1 a_1 + \dots + m_n a_n) = 0
\end{equation}
if and only if
\begin{equation}\label{day}
\phi( m_1 a_1 + \dots + m_n a_n + m \vec c) = 0
\end{equation}
for an integer $m$ with $|m| \leq n^{C_1 n/2} \frac{N_r}{|c_r|}$, and furthermore this choice of $m$ is unique.

For $m_1,\dots,m_n$ in the range $|m_i| \leq n^{C_1 n/10}$, we can repeat the proof of Lemma \ref{none-vanish} to conclude that \eqref{chy} holds if and only if $m_1 a_1 + \dots + m_n a_n$ is a multiple of $a_n - a_1$, or equivalently (by \eqref{aa} and the fact that the $c_j$ have no common factor) an integer multiple of $\vec c$.  Thus we have 
$$ m_1 a_1 + \dots + m_n a_n + m \vec c = 0$$
(and hence \eqref{day}) for some integer $m$.  Since the $m_i$ have magnitude at most $n^{C_1 n/10}$, each $a_i$ has the $r^{\operatorname{th}}$ coefficient of magnitude at most $N_r$, we see (for $C_1$ large enough) that $m$ has magnitude at most $n^{C_1 n/2} \frac{N_r}{|c_r|}$.
Applying $\phi$, we see that \eqref{chy} implies \eqref{day}.

Conversely, suppose that $m_1,\dots,m_n$ are integers in the range $|m_i| \leq n^{C_1 n/10}$ such that \eqref{day} holds for some integer $m$ with $|m| \leq n^{C_1 n/2} \frac{N_r}{|c_r|}$.  For $j=1,\dots,r$, the $j^{\operatorname{th}}$ coefficient of $m_i a_i$ for $i=1,\dots,n$ has magnitude at most $n^{C_1 n/10} N_j$, while from \eqref{cjrp}, the corresponding coefficient of $mc'$ has magnitude at most $n^{C_1 n/2} N_j$.  Summing, we see that the $j^{\operatorname{th}}$ coefficient of $m_1 a_1 + \dots + m_n a_n + m \vec c$ has magnitude at most $n^{C_1 n} N_j$ if $C_1$ is large enough.  As $Q$ is $n^{C_1 n}$-proper, we conclude from \eqref{day} that
\begin{equation}\label{ido}
 m_1 a_1 + \dots + m_n a_n + m \vec c = 0.
\end{equation}
Thus $m_1 a_1 + \dots + m_n a_n$ is a multiple of $a_n - a_1$, and on applying $\phi'$ we conclude \eqref{chy}.  Note that the identity \eqref{ido} also shows that the choice of $m$ is unique.
\end{proof}

Applying the above lemma to \eqref{ideo}, we conclude that
$$ \int_{\R/\Z} D(t \phi(\vec c)) \prod_{i=1}^{n} \eta(t v_i)\ dt \gg n^{-O(n)}.$$
By Lemma \ref{lam} (and crudely bounding $D$ by $O( n^{C_1 n} )$), the contribution to this integral of those $t$ for which
$$  \min_{1 \leq i \leq n} \| t v_i \|_{\R/\Z} \leq \frac{1}{n+1}$$
is $O( n^{-99 C_1 n} )$.  For $C_1$ large enough, this implies that we must have
$$  \min_{1 \leq i \leq n} \| t v_i \|_{\R/\Z} > \frac{1}{n+1}$$
for at least one value of $t$, and \eqref{jip} follows.

\section{Velocities in a short progression}\label{short-sec}

We now prove Propositions \ref{short}, \ref{short2}.

The key lemma in proving Proposition \ref{short} is the following.

\begin{lemma}\label{lo} Let $n, k \geq 1$ be natural numbers, and let $v_1,\dots,v_n$ be positive integers with $v_i \leq kn$ for all $i=1,\dots,n$, and such that $\delta(v_1,\dots,v_n) < \frac{1}{n+1}$.
\begin{itemize}
\item[(i)]  If $1 \leq j \leq n+1$, then at least one of the $v_1,\dots,v_n$ is a multiple of $j$.
\item[(ii)]  If $1 \leq j \leq n$ and $a$ is coprime to $j$, then there exists $v_i, i=1,\dots,n$ such that either $v_i = cj$ for some $c=1,\dots,k-1$, or else $v_i = a\ \operatorname{mod}\ j$ and $v_i > k(n+1-j)$.
\end{itemize}
\end{lemma}

\begin{proof}  By hypothesis, for every $t \in \R/\Z$ there exists $1 \leq i \leq n$ such that $\|tv_i\|_{\R/\Z} < \frac{1}{n+1}$.  Applying this claim with $t \coloneqq \frac{1}{j}$ for some $1 \leq j \leq n+1$, we obtain (i).

Now we prove (ii).  Since $a$ is coprime to $j$, we can find an integer $d$ coprime to $j$ such that $ad = -1\ \operatorname{mod}\ j$.  We apply the hypothesis with $t \coloneqq \frac{d}{j} + \frac{1}{kj(n+1)}$, and conclude that there exists $v_i, i=1,\dots,n$, such that
$$ \left\|\frac{v_i d}{j} + \frac{v_i}{kj(n+1)} \right\|_{\R/\Z} < \frac{1}{n+1}.$$
We divide into cases depending on the residue class of $v_i d$ modulo $j$.  If $v_i d = 0\ \operatorname{mod}\ j$, then since $\frac{v_i}{kj(n+1)}$ is positive and bounded above by $\frac{kn}{kj(n+1)} \leq 1 - \frac{1}{n+1}$, we have
$$ \frac{v_i}{kj(n+1)} < \frac{1}{n+1}$$
and hence $v_i < kj$.  On the other hand, as $d$ is coprime to $j$ and $v_i d = 0\ \operatorname{mod}\ j$, $v_i$ must be a multiple of $j$.  Thus $v_i = cj$ for some $c=1,\dots,k-1$.  This already covers the $j=1$ case, so we now may assume $j > 1$.

Now suppose that $v_i d = -1\ \operatorname{mod}\ j$.  Then we must have
$$ \frac{v_i}{kj(n+1)} > \frac{1}{j} - \frac{1}{n+1}$$
and hence
$$ v_i > k (n+1-j);$$
also, since $ad = -1\ \operatorname{mod}\ j$, we must also have $v_i = a\ \operatorname{mod}\ j$.

Finally, suppose $v_i d\ \operatorname{mod}\ j$ is not equal to $0$ or $-1$.  Then we must have
$$ \frac{v_i}{kj(n+1)} > \frac{2}{j} - \frac{1}{n+1}$$
and hence
$$ v_i > k (2n+2-j) \geq nk,$$
contradicting the hypothesis.
\end{proof}

Now we can prove Proposition \ref{short}.  Suppose for contradiction that we can find $n \geq 1$ and positive integers $v_1,\dots,v_n \leq 1.2 n$ such that $\delta(v_1,\dots,v_n) < \frac{1}{n+1}$.  
From Lemma \ref{lo}(i) we see that for any $0.6 n < j \leq n+1$, some multiple of $j$ must lie in $\{v_1,\dots,v_n\}$; since $2j > 1.2 n$, we conclude that
$$ j \in \{v_1,\dots,v_n\}$$
whenever $0.6 n < j \leq n+1$.

Next, suppose that $1 \leq j \leq 0.4n + 1$.  Applying Lemma \ref{lo}(ii) with $k=2$ and $a=1$, we see that there exists $v_i, i=1,\dots,n$ which is either equal to $j$, or is at least $2(n+1-j) \geq 1.2 n$.  The latter is impossible, hence we have
$$ j \in \{v_1,\dots,v_n\}$$
whenever $1 \leq j \leq 0.4 n + 1$.

Now suppose that $\frac{n+1}{2} < j \leq 0.6n$.  We apply Lemma \ref{lo}(ii) with $k=3$ and $a=1$ to conclude that there exists $v_i, i=1,\dots,n$ which is either equal to $j$ or $2j$, or is at least $3(n+1-j) > 1.2 n$.  The latter case cannot occur, and hence
$$ j \hbox{ or } 2j  \in \{v_1,\dots,v_n\}$$
whenever $\frac{n+1}{2} < j \leq 0.6n$.  

Finally, suppose that $0.4 n + 1 < j \leq \frac{n+1}{2}$.  We apply Lemma \ref{lo}(ii) with $k=2$ and $a=j-1$ to conclude that there exists $v_i, i=1,\dots,n$ which is either equal to $j$, or is at least $2(n+1-j) \geq n+1$ and is equal to $-1\ \operatorname{ mod } j$.  Since $3j-1 > 1.2 n$ and $j-1, 2j-1 < n+1$, we thus have
$$ j \in \{v_1,\dots,v_n\}$$
for $0.4 n + 1 < j \leq \frac{n+1}{2}$.

Observe that each of the above conditions places exactly one element (either $j$ or $2j$) in $\{v_1,\dots,v_n\}$ for $j=1,\dots,n+1$, and these elements are all distinct (if $j$ lies in the range $\frac{n+1}{2} < j \leq 0.6n$, then $2j > n+1$ and so this element does not collide with any of the others).  We conclude that $\{v_1,\dots,v_n\}$ has cardinality at least $n+1$, which is absurd.  This completes the proof of Proposition \ref{short}.

Now we prove Proposition \ref{short2}.  We allow implied constants to depend on $C$, thus $v_i = O(n)$ for all $i=1,\dots,n$. We will adapt the arguments following Proposition \ref{will}, except that we will use ``small primes'' rather than ``medium primes''.

We will need a small quantity $0 < \eps < 1$, depending only on $C$, to be chosen later.  Define a \emph{small prime} to be a prime $p$ between $\exp(1/\eps)$ and $\exp(1/\eps^2)$.  The number of positive integers between $1$ and $Cn$ that are not divisible by any small prime is
$$ (1+o(1)) Cn \prod_p (1-\frac{1}{p})$$
where $p$ ranges over small primes (and the $o(1)$ notation is with respect to the limit $n \to \infty$, holding $C$ and $\eps$ fixed); by Mertens' theorem, this expression is $O(\eps C n )$.  Thus, if $\eps$ is small enough, we see that $\asymp n$ of the $v_i$ will have at least one small prime factor.

Call an integer \emph{bad} if it is divisible by the square of a small prime, or by two small primes $p,p'$ with $p < p' \leq (1+\eps^2) p$, and good otherwise.  The number of bad integers between $1$ and $Cn$ can be bounded by
$$ \ll C n \sum_p \sum_{p \leq p' \leq (1+\eps^2) p} \frac{1}{p p'}.$$
One can crudely bound the inner sum by $O( \frac{\eps^2}{p})$, and then by another application of Mertens' theorem, the total number of bad integers is $O( \eps C n )$.  Thus, again if $\eps$ is small enough, we see that $\asymp n$ of the $v_i$ will be good and have at least one small prime factor.  Removing the integers of size $\eps n$, we may thus locate a subset $K$ of $\{1,\dots,n\}$ of cardinality $\# K \asymp n$, such that for each $k \in K$, $v_k$ is a good integer between $\eps n$ and $Cn$ with at least one small prime factor.

For each $k \in K$, we may factor
$$ v_k = p_k v'_k$$
where $p_k$ is the minimal small prime dividing $v_k$, thus all the small primes dividing $v'_k$ are larger than $(1+\eps^2) p_k$.  For each $v_k$, we associate the set $S(v_k)$ of integers of the form $p'_k v'_k$, where $p'_k$ is a small prime between $(1+\eps^2)^{-1} p_k$ and $(1+\eps^2) p_k$.  Clearly $p'_k$ will be the minimal small prime dividing $p'_k v'_k$; in particular, $v'_k$ can be determined from any element of $S(v_k)$.  Thus, if $S(v_k)$ and $S(v_{k'})$ intersect, then we have $v'_k = v'_{k'}$.

From the prime number theorem, we see that each set $S(v_k)$ has cardinality $\gg \eps^2 \frac{p_k}{\log p_k} \gg \frac{1}{\eps}$.  On the other hand, as each $v_k$ is of size at most $Cn$, each element of $S(v_k)$ is of size $O(Cn)$.  We conclude that the number of $k$ for which $S(v_k)$ does not intersect any other $S(v_{k'})$ cannot exceed $O( \eps C n )$.  For $\eps$ small enough, we thus can find a subset $K'$ of $K$ of cardinality $\# K' \asymp n$, such that for each $k \in K'$, there is another $k' \in K'$ such that $S(v_k)$ intersects $S(v_{k'})$.  By the preceding discussion, this implies that $v'_k = v'_{k'}$.

As in Section \ref{first-sec}, we define the quantity $A$ by requiring
$$ \delta(v_1,\dots,v_n) = \frac{1}{2n} + \frac{A}{n^2}$$
and introduce the multiplicity function
$$ F \coloneqq  \sum_{i=1}^n 1_{B(v_i,\delta_n)}(t).$$
By repeating the proof of \eqref{rod}, we have
\begin{equation}\label{rod2}
\frac{2A}{n} \geq \frac{1}{2} \int_{\R/\Z} F(t) 1_{F(t) \geq 2}\ dt
\end{equation}
and by repeating the proof of \eqref{dor} we have
$$
\int_{\R/\Z} F(t) 1_{F(t) \geq 2}\ dt \geq \sum_{k \in K'} m\left( B(v_k;\delta_n) \cap \bigcup_{l \in K': l \neq k} B(v_l;\delta_n)\right ).$$
For each $k \in K'$, we see from previous discussion that there is an $l \in K'$ distinct from $k$ such that $v'_k = v'_l$, hence $v_k = p_k v'_k$ and $v_l = p_l v'_k$.  Hence
$$ B(v_k,v_l;\delta_n,\delta_n) \supset B\left(v'_k, \frac{\delta_n}{p_k p_l}\right )$$
and thus (by Corollary \ref{clb} or \eqref{mvd})
$$ m\left( B(v_k;\delta_n) \cap \bigcup_{l \in K': l \neq k} B(v_l;\delta_n)\right ) \gg \frac{\delta_n}{p_k p_l} \gg_\eps \frac{1}{n}.$$
Since $\# K' \asymp n$, we conclude that
$$
\int_{\R/\Z} F(t) 1_{F(t) \geq 2}\ dt \gg_\eps 1$$
and hence by \eqref{rod2} we have $A \gg_\eps n$, and the claim follows.


\begin{thebibliography}{10}

\bibitem{alon}
N. Alon, \emph{The chromatic number of random Cayley graphs}, European J. Combin. \textbf{34} (2013), no. 8, 1232--1243.

\bibitem{bs}
J. Barajas, O. Serra, \emph{The lonely runner with seven runners}, Electron. J. Combin. \textbf{15} (2008), R48, 18 pp..

\bibitem{bs2}
J. Barajas, O. Serra, \emph{On the chromatic number of circulant graphs}, Discrete Math.
\textbf{309} (2009), 5687--5696.

\bibitem{bggst}
W. Bienia, L. Goddyn, P. Gvozdjak, A. Seb\"o, M. Tarsi, \emph{Flows, view obstructions
and the lonely runner}, J. Combin. Theory Ser. B \textbf{72} (1998), 1--9.

\bibitem{bilu}
Y. Bilu, \emph{Structure of sets with small sumset}, Structure theory of set addition. 
Ast\'erisque No. \textbf{258} (1999), xi, 77--108. 

\bibitem{bhk}
T. Bohman, R. Holzman, D. Kleitman, \emph{Six lonely runners}, In honor of Aviezri Fraenkel on the occasion of his 70th birthday. 
Electron. J. Combin. \textbf{8} (2001), no. 2, Research Paper 3, 49 pp.

\bibitem{chen}
Y. G. Chen, \emph{View-obstruction problems in $n$-dimensional Euclidean space and a generalization of them},
Acta Math. Sinica \textbf{37} (1994), no. 4, 551--562.

\bibitem{cc}
Y. G. Chen, T. W. Cusick, \emph{The view-obstruction problem for $n$-dimensional cubes}, J. Number Theory
\textbf{74} (1999), no. 1, 126--133.


\bibitem{cusick}
T. W. Cusick, \emph{View obstruction problems}, Aequationes Math. \textbf{9} (1973), 165--170.

\bibitem{czer}
S. Czerwi\'nski, \emph{Random runners are very lonely}, Journal of Combinatorial Theory, Series A \textbf{119} (2012),
no. 6, 1194--1199.

\bibitem{cg}
S. Czerwi\'nski, J. Grytczuk, \emph{Invisible runners in finite fields}, Inf. Proc. Lett. \textbf{108} (2008), 64--67.

\bibitem{dub}
A. Dubickas, \emph{The lonely runner problem for many runners}, Glas. Mat. Ser. III \textbf{46}(66) (2011), no. 1, 25–30. 

\bibitem{er}
P. Erd\H{o}s, R. Rado, \emph{Intersection theorems for systems of sets}, Journal of the London Mathematical Society, Second Series, \textbf{35} (1960), 85--90.

\bibitem{gw}
L. Goddyn, E. B. Wong, \emph{Tight instances of the lonely runner}, Integers \textbf{6} (2006), A38.

\bibitem{green}
B. Green, \emph{Notes on progressions and convex geometry}, available at {\tt http://people.maths.ox.ac.uk/greenbj/papers/convexnotes.pdf}

\bibitem{ik}
H. Iwaniec, E. Kowalski, Analytic Number Theory.  Colloquium Publications Vol. 53, American Mathematical Society, 2004.

\bibitem{pandey}
R. K. Pandey, \emph{A note on the lonely runner conjecture}, J. Integer Seq. \textbf{12} (2009), Article 09.4.6, 4 pp..

\bibitem{ps}
G. Perarnau, O. Serra, \emph{Correlation among runners and some results on the lonely runner conjecture}, Electron. J. Combin. \textbf{23} (2016), no. 1, Paper 1.50, 22 pp. 

\bibitem{rtv}
I. Ruzsa, Zs. Tuza, M. Voigt, \emph{Distance graphs with finite chromatic number}, J.
Combin. Theory Ser. B \textbf{85} (2002), 181--187.

\bibitem{tv}
T. Tao, V. Vu, Additive Combinatorics, Cambridge University Press, 2006.

\bibitem{wills}
J. M. Wills, \emph{Zwei S\"atze \"uber inhomogene diophantische Approximation von Irrationalzahlen}, Monatsch. Math. \textbf{61} (1967), 263--269.

\bibitem{zhu}
X. Zhu, \emph{Circular chromatic number of distance graphs with distance sets of cardinality $3$}, J. Graph Theory \textbf{41} (2002) 195--207.

\end{thebibliography}
\end{document}